\documentclass{amsart}

\usepackage{latexsym}
\usepackage{amssymb}
\usepackage{amsmath,mathabx}

\newtheorem{theorem}{Theorem}[section]
\newtheorem{lemma}[theorem]{Lemma}
\newtheorem{proposition}[theorem]{Proposition}
\newtheorem{corollary}[theorem]{Corollary}

\theoremstyle{definition}
\newtheorem{definition}[theorem]{Definition}

\theoremstyle{remark}
\newtheorem{remark}[theorem]{Remark}

\numberwithin{equation}{section}

\newcommand{\C}{\mathbb{C}}
\newcommand{\R}{\mathbb{R}}

\newcommand{\re}{\operatorname{Re}}

\newcommand{\f}{\varphi}

\newcommand{\Rr}{\mathcal{R}}
\newcommand{\ep}{\varepsilon}

\newcommand{\B}{\mathcal B}
\newcommand{\F}{\mathcal F}
\newcommand{\Om}{\Omega}
\newcommand{\Hh}{\mathcal H}
\newcommand{\Cc}{\mathcal C}

\newcommand{\om}{\omega}
\newcommand{\T}{\mathcal T}

\begin{document}

\title[Property (T) and infinite measures]{Property (T) and actions on infinite measure spaces}

\author{Paul Jolissaint}
\address{Universit\'e de Neuch\^atel,
       Institut de Math\'ematiques,       
       E.-Argand 11,
       2000 Neuch\^atel, Switzerland}
       
\email{paul.jolissaint@unine.ch,}


\date{\today}


\begin{abstract}
The aim of the article is to provide a characterization of Kazhdan's property (T) for locally compact, second countable pairs of groups $H\subset G$ in terms of actions on infinite, $\sigma$-finite measure spaces. It is inspired by the recent characterization of the Haagerup property by similar actions due to T. Delabie, A. Zumbrunnen and the author.
\end{abstract}

\maketitle

\section{Introduction}

Throughout this article, $G$ denotes a locally compact, second countable group (lcsc group for short); we assume furthermore that it is non-compact.

In \cite{DJZ}, the author and his co-authors T. Delabie and A. Zumbrunnen present a characterization of the Haagerup property in terms of actions on infinite, $\sigma$-finite measure spaces having an invariant mean and whose associated permutation representations are $C_0$.

Citing A. Valette in Chapter 7 of \cite{ccjjv}, "According to the philosophy that, to any characterization of property (T) there is a parallel characterization of the Haagerup property", the aim of the present note is to propose a sort of reciprocal to that statement, namely, to propose a new characterization of property (T) inspired by the above mentionned one of the Haagerup property. 

\par\vspace{2mm}
The objects under study here are what we call \textit{dynamical systems}: given a lcsc group $G$, such a dynamical system is a quadruple $(\Om,\B,\mu,G)$ where $(\Om,\B,\mu)$ is a measure space on which $G$ acts by $\mu$-preserving automorphisms. For brevity, if $\mu$ is infinite, we denote by $\B_f$ the subset of elements $B\in\B$ such that $0\leq \mu(B)<\infty$. 

Throughout the article, if $(\Om,\B,\mu,G)$ is a dynamical system as above, we denote by $\pi_\Om:G\rightarrow U(L^2(\Om,\B,\mu))$ the associated \textit{permutation representation} defined by
\[
(\pi_\Om(g)\xi)(\om):=\xi(g^{-1}\om)
\]
for $g\in G$ and $\om\in\Om$.

Let $(\pi,\Hh)$ be a unitary representation of the lcsc group $G$. Then recall from Definition 1.1.1 of \cite{BHV} that $\pi$ \textit{almost has invariant vectors} if, for every compact set $\emptyset\not=Q\subset G$ and for every $\ep>0$, there exists a unit vector $\xi\in\Hh$ such that
\[
\sup_{g\in Q}\Vert\pi(g)\xi-\xi\Vert<\ep.
\] 

Recall also from Definition 1.1 of \cite{JolT} that, if $H$ is a closed subgroup of a lcsc group $G$, then \textit{the pair} $H\subset G$ \textit{has Property (T)} if, for every unitary representation $\pi$ of $G$ which almost has invariant vectors, there exists a vector $\xi\not=0$ such that $\pi(h)\xi=\xi$ for every $h\in H$. In particular, $G$ has property (T) if the pair $G\subset G$ has property (T).

\par\vspace{2mm}

Before stating our main result, we need two definitions. The first one is a weakening of $C_0$-\textit{dynamical systems} from \cite{DJZ}: recall from Definition 1.2 of \cite{DJZ} that a dynamical system $(\Om,\B,\mu,G)$ is $C_0$ if, for all $A,B\in\B_f$, one has $\lim_{g\to\infty}\mu(gA\cap B)=0$.

\begin{definition}\label{wC0}
Let $(\Om,\B,\mu,G)$ be a dynamical system as above. We say that it is \textit{weakly} $C_0$ if, for all $A,B\in\B_f$ and for any $\ep>0$, there exists $g\in G$ such that $\mu(gA\cap B)<\ep$.
\end{definition}

Next, the following type of sequence $(B_m)_{m\geq 1}\subset \B_f$ plays a crucial role in the proof of the main result of \cite{DJZ} (proof of Proposition 2.8):

\begin{definition}\label{almostinv}
Let $(\Om,\B,\mu,G)$ be a dynamical system. A sequence $(B_m)_{m\geq 1}\subset \B_f$ is said to be \textit{almost invariant} if $\mu(B_m)=1$ for every $m\geq 1$ and if, for every compact set $K\subset G$, one has
\[
\lim_{m\to\infty}\sup_{g\in K}\mu(gB_m\cap B_m)=1.
\]
\end{definition}

\begin{remark}\label{rem1.3}
Let $(\Om,\B,\mu,G)$ be a weakly $C_0$ dynamical system.
\begin{enumerate}
\item [(1)] If $A\in\B_f$ is $G$-invariant then $\mu(A)=0$. In particular, a non-trivial weakly $C_0$ dynamical system has an infinite measure.
\item [(2)] As will be explained in more details in Section 2, the unitary representation $\pi_\Om$ is \textit{weakly mixing} in the sense of \cite{BR}, which means, by Theorem 1.9 of \cite{BR}, that $\pi_\Om$ has no non-trivial finite-dimensional subrepresentation.
\end{enumerate} 
\end{remark}

\begin{remark}\label{rem1.4}
Let $(\Om,\B,\mu,G)$ be a dynamical system which contains an almost invariant sequence of sets $(B_m)_{m\geq 1}\subset \B_f$. Then the representation $(\pi_\Om,L^2(\Om,\B,\mu))$ almost has invariant vectors namely, $\xi_m:=\chi_{B_m}\in L^2(\Om,\B,\mu)$ is a unit vector for every $m\geq 1$ and, for every compact set $Q\subset G$ and every $\ep>0$, one has
\[
\sup_{g\in Q}\Vert \pi_\Om(g)\xi_m-\xi_m\Vert<\ep
\]
for every large enough $m$.\\
Indeed, one has for all $g$ and $m$
\begin{align*}
\Vert \pi_\Om(g)\xi_m-\xi_m\Vert^2
&=
2(1-\re\langle\pi_\Om(g)\xi_m|\xi_m\rangle)\\
&=
2(1-\int_\Om \chi_{gB_m}(\om)\chi_{B_m}(\om)d\mu(\om))\\
&=
2(1-\mu(gB_m\cap B_m))
\end{align*}
which converges to $0$ uniformly on $Q$ as $m\to\infty$.
\end{remark}

Here is our main result.

\begin{theorem}\label{mainthm}
Let $G$ be a lcsc group and let $H\subset G$ be a closed subgroup of $G$.
\begin{enumerate}
\item [(1)] If the pair $H\subset G$ has Property (T) and if $(\Om,\B,\mu,G)$ is a dynamical system such that $(\Om,\B,\mu,H)$ is weakly $C_0$, then $(\Om,\B,\mu,G)$ admits no almost invariant sequence.
\item [(2)] If the pair $H\subset G$ does not have Property (T), then there exists a $\sigma$-finite, dynamical system $(\Om,\B,\mu,G)$ which has an almost invariant sequence and whose restriction $(\Om,\B,\mu,H)$ is weakly $C_0$.
\end{enumerate}
\end{theorem}

The proof of Theorem \ref{mainthm} will be given in Section 3, and it rests partly on weakly mixing representations and weakly mixing actions of groups on probability measure-preserving spaces, which is the subject of the next section.

\section{Weakly mixing representations, weakly mixing actions}

As mentioned in Section 1, we will see that the notion of weakly $C_0$ dynamical systems is closely related to weakly mixing representations and weakly mixing actions.
\par\vspace{2mm}
 
Our references for the latter properties are the article \cite{BR} of V. Bergelson and J. Rosenblatt on the one hand, and Chapters 1 and 2 of the monograph \cite{Glas} by E. Glasner on the other hand.

\par\vspace{3mm}\noindent
\textbf{Notation}
Let $G$ be a lcsc group and let $(S,\B_S,\nu)$ be a probability space on which $G$ acts by Borel automorphisms and preserves $\nu$. Then the subset 
\[
L^2_0(S,\nu):=\left\{\xi\in L^2(S,\nu) \colon \int_S\xi d\nu=0\right\}
\] 
is a $G$-invariant closed subspace of $L^2(S,\nu)$, and we denote by $\pi_{S,0}$ the restriction of $\pi_S$ to $L^2_0(S,\nu)$.

\begin{definition}
Let $G$ be a lcsc group.
\begin{enumerate}
\item [(1)] A unitary representation $(\pi,\Hh)$ of $G$ is \textit{weakly mixing} if it contains no non-trivial finite-dimensional subrepresentation.
\item [(2)] Let $(S,\B_S,\nu,G)$ be a dynamical system where $\nu$ is a $G$-invariant probability measure. We say that the action of $G$ on $S$ is \textit{weakly mixing} if the representation $\pi_{S,0}$ is a weakly mixing representation.
\end{enumerate}
\end{definition} 

\begin{remark}
The original definition of weakly mixing representations involve the space of weakly almost periodic functions $WAP(G)$ on $G$ and the unique invariant mean on it (cf. \cite{BR}, Definition 1.1, and \cite{Glas}, Definition 3.2), and the characterization in terms of finite-dimensional subrepresentations is for instance Theorem 1.9 of \cite{BR}.
We have chosen not to introduce $WAP(G)$ because we feel that it is useless in the present context.
\end{remark}

We gather some characterisations of weakly mixing representations that will be used in the next section.

\begin{proposition}
Let $G$ be a lcsc group, let $(\pi,\Hh)$ be a unitary representation of $G$ and let $\T$ be a total subset of $\Hh$. The following conditions are equivalent:
\begin{enumerate}
\item [(a)] $(\pi,\Hh)$ is a weakly mixing representation of $G$;
\item [(b)] for every $\ep>0$ and for all $\xi_1,\ldots,\xi_m\in\Hh$, there exists $g\in G$ such that $|\langle\pi(g)\xi_j|\xi_j\rangle|<\ep$ for all $j=1,\ldots,m$;
\item [(c)] for every $\ep>0$ and for all $\xi_1,\ldots,\xi_m\in\T$, there exists $g\in G$ such that $|\langle\pi(g)\xi_j|\xi_k\rangle|<\ep$ for all $j,k=1,\ldots,m$.
\end{enumerate}
\end{proposition}
\begin{proof} Equivalence between (a) and (b) follows from Corollary 1.6 and from Theorem 1.9 of \cite{BR}, and equivalence between (b) and (c) is a consequence of density of the span of $\T$ and of polar decomposition of scalar products in Hilbert spaces; namely, 
\begin{align*}
\langle\xi|\eta\rangle
={}&
\frac{1}{4}(\langle\xi+\eta|\xi+\eta\rangle-\langle\xi-\eta|\xi-\eta\rangle\\
&
+i\langle\xi+i\eta|\xi+i\eta\rangle-i\langle\xi-i\eta|\xi-i\eta\rangle).
\end{align*}
for all $\xi,\eta\in \Hh$. 
\end{proof}

\begin{corollary}
Let $G$ be a lcsc group and let $(\Om,\B,\mu,G)$ be a dynamical system. Then the following conditions are equivalent:
\begin{enumerate}
\item [(a)] $(\Om,\B,\mu,G)$ is weakly $C_0$;
\item [(b)] the permutation representation $\pi_\Om$ is weakly mixing.
\end{enumerate}
\end{corollary}
\begin{proof}
Assume that condition (a) holds. This means that for all $A,B\in \B_f$, and for every $\ep>0$, there exists $g\in G$ such that $|\langle\pi_\Om(g)\chi_A|\chi_B\rangle|<\ep$. If $A_1,\ldots,A_m\in\B_f$ and if $\ep>0$, then considering $A=B=\bigcup_{j=1}^m A_j$ which belongs to $\B_f$, there exists $g\in G$ such that $|\langle\pi_\Om(g)\chi_{A_j}|\chi_{A_k}\rangle|<\ep$ for all $j,k=1,\ldots,m$. As linear combinations of characteristic functions $\chi_A$ with $A\in B_f$ are dense in $L^2(\Om,\B,\mu)$, this proves that condition (b) holds.\\
Conversely, if (b) holds, then (a) is the special case $m=2$ with $\xi_1=\chi_{A}$ and $\xi_2=\chi_B$ in condition (c) of Proposition 2.3.
\end{proof}

\begin{lemma}
Let $(S,\B_S,\nu)$ be a probability space on which $G$ acts by $\nu$-preserving Borel automorphisms, and assume that the action is weakly mixing. Then, for every $\ep>0$ and for all $A_1,\ldots,A_m\in \B_S$, there exists $g\in G$ such that
\[
|\nu(gA_j\cap A_k)-\nu(A_j)\nu(A_k)|<\ep
\]
for all $j,k=1,\ldots,m$.
\end{lemma}
\begin{proof}
For $j=1,\ldots,m$, set $\xi_j:=\chi_{A_j}-\nu(A_j)$. Then $\xi_j\in L^2_0(S,\nu)$ for every $j$, and, for every $\ep>0$, by Proposition 2.3, there exists $g\in G$ such that $|\langle\pi_{S,0}(g)\xi_j|\xi_k\rangle|<\ep$ for all $j,k=1,\ldots,m$. But we have
\begin{align*}
\langle\pi_{S,0}(g)\xi_j|\xi_k\rangle
&=
\langle\chi_{gA_j}-\nu(A_j)|\chi_{A_k}-\nu(A_k)\rangle\\
&=
\nu(gA_j\cap A_k)-2\nu(A_j)\nu(A_k)+\nu(A_j)\nu(A_k)\\
&=
\nu(gA_j\cap A_k)-\nu(A_j)\nu(A_k)
\end{align*}
for all $j,k=1,\ldots,m$.
\end{proof}

\section{Proof of Theorem \ref{mainthm}}

Let $H\subset G$ be a pair of lcsc groups with Property (T) as in statement (1) of Theorem \ref{mainthm}, and let $(\Om,\B,\mu,G)$ be a  dynamical system whose restriction to $(\Om,\B,\mu,H)$ is weakly $C_0$. Corollary 2.4 implies that the permutation representation $\pi_\Om$ of $H$ is weakly mixing, hence that it does not have any non-trivial finite-dimensional subrepresentation. If there existed an almost invariant sequence of sets $(B_m)\subset \B_f$ as in Definition \ref{almostinv}, then the unitary representation $\pi_\Om$ of $G$ would almost have invariant vectors, hence there would exist a non-zero vector $\xi$ such that $\pi_\Om(h)\xi=\xi$ for every $h\in\Hh$ by Property (T), but this contradicts the weakly mixing property of $\pi_\Om$ restricted to $H$. Hence $(\Om,\B,\mu,G)$ has no almost invariant sequence.

\par\vspace{2mm}
The rest of the present section is devoted to the proof of statement (2) of Theorem \ref{mainthm}. Thus we assume henceforth that the pair $H\subset G$ does not have Property (T). 
Let us choose an increasing sequence of compact subsets $(K_n)_{n\geq 1}$ of $G$ with the following properties: $e\in \mathring{K}_1$, $K_n\subset \mathring{K}_{n+1}$ for every $n\ge 1$ and $G=\bigcup_{n\geq 1}K_n$. Hence, $K_1$ is a compact neighbourhood of $e$, and for every compact set $K\subset G$, there exists $m$ such that $K\subset K_m$.

\begin{lemma}
With the above hypotheses, there exists a conditionally negative definite function $\psi:G\rightarrow\R_+$ whose restriction to $H$ is unbounded.
\end{lemma}
\begin{proof}
By Theorem 1.2 of \cite{JolT}, there exists a sequence $(\psi_n)_{n\geq 1}$ of real-valued, positive definite and normalized functions on $G$ which converges to $1$ uniformly on compacts sets, but such that
\[
\sup_{h\in H}|\psi_n(h)-1|\not\to 0.
\] 
Extracting subsequences if necessary, we assume that there exists $c>0$ and a sequence $(h_n)_{n\geq 1}\subset H$ such that $1-\psi_n(h_n)\geq c$
and
\[
\sup_{g\in K_n}|\psi_n(g)-1|\leq \frac{1}{n^2}
\]
for all $n$. Then set 
\[
\psi=\sum_{n\geq 1}\sqrt{n}(1-\psi_n).
\]
It defines a conditionally negative definite function on $G$ which satisfies 
\[
\psi(h_k)=\sum_{n\not=k}\sqrt{n}(1-\psi_n(h_k))+\sqrt{k}(1-\psi_k(h_k))\geq \sqrt{k}c.
\]
Hence $\psi$ is unbounded on $H$.
\end{proof}

An adaptation of Proposition 2.2.3 of \cite{ccjjv}, of \cite{CW} and of Theorem A.1 of \cite{GW} shows that there exists a measure-preserving $G$-action on a standard probability space $(S,\B_S,\nu)$ with the following properties:
\begin{enumerate}
\item [(a)] the restriction to $H$ of the action is weakly mixing (this is where the existence of $\psi$ in Lemma 3.1 is needed; see details below);
\item [(b)] there exists a \textit{non-trivial asymptotically invariant sequence} of Borel subsets of $S$, namely, there exists a sequence $(A_n)_{n\geq 1}\subset\B_S$ such that $\nu(A_n)=1/2$ for every $n$ and such that, for every compact set $K\subset G$, 
\[
\lim_{n\to\infty}\sup_{g\in K}\nu(gA_n\bigtriangleup A_n)=0
\]
where $A\bigtriangleup B=A\setminus B\cup B\setminus A$ for all sets $A,B$.
\end{enumerate}
Furthermore, the proof of Lemma 1.3 of \cite{AEG} shows that we can (and will) assume that $S$ is a compact metric space on which $G$ acts continuously, and that $\nu$ has support $S$.

\par\vspace{2mm}
We think that it is helpful to describe the construction of the probability $G$-space $(S,\B_S,\nu)$ with some details. We follow faithfully the proof of Theorem 2.2.2 of \cite{ccjjv}.

By Lemma 3.1, let $\psi:G\rightarrow\R_+$ be a conditionally negative definite function which is unbounded on $H$. For $n\geq 1$, we set $\f_n=\exp(-\psi/n)$, and we denote by $(\pi_n,\Hh_n,\xi_n)$ the associated Gel'fand-Naimark-Segal triple. Since $\f_n$ is real-valued, there exists a real Hilbert subspace $\Hh_n'$ of $\Hh_n$, containing $\xi_n$, such that
\[
\Hh_n=\Hh_n'\oplus i\Hh_n'\quad\textrm{and}\quad
\pi_n(g)\Hh_n'=\Hh_n'
\]
for all $n$ and $g$. We set $\Hh=\bigoplus_{n\geq 1}\Hh_n$, $\Hh'=\bigoplus_{n\geq 1}\Hh_n'$ so that $\Hh=\Hh'\oplus i\Hh'$, and $\pi=\bigoplus_{n\geq 1}\pi_n$. Finally, we identify $\xi_n$ with the corresponding vector 
\[
0\oplus\ldots\oplus\xi_n\oplus\ldots \in \Hh'
\]
and we observe that $\xi_n\perp\xi_m$ when $n\not=m$.

Next, set $\Hh^\sigma=\bigoplus_{k\geq 0}\Hh^{\times k}$, where $\Hh^{\times 0}=\C$, and for $k>0$, $\Hh^{\times k}$ is the $k$-th symmetric tensor product of $\Hh$, that is, the closed subspace of the Hilbert tensor product space $\Hh^{\otimes k}$ generated by the vectors of the form
\[
\sum_{s\in S_k}\eta_{s(1)}\otimes\ldots\otimes\eta_{s(k)}
\]
where $S_k$ denotes the usual permutation group. Then the representation $\pi$ extends in a natural way to a representation $\pi^\sigma$ of $G$ on $\Hh^\sigma$ which leaves the subspace $\Hh_0^\sigma=\Hh^\sigma\ominus \Hh^{\times 0}$ invariant. Finally, we denote by $\pi_0^\sigma$ the restriction of $\pi^\sigma$ to $\Hh_0^\sigma$.

\begin{lemma}
The restriction of $\pi_0^\sigma$ to $H$ is weakly mixing.
\end{lemma}
\begin{proof}
Let $(h_\ell)_{\ell\geq 1}\subset H$ be a sequence such that $\psi(h_\ell)\to\infty$ as $\ell\to\infty$. Then, as in the proof of Lemma 2.1 of \cite{JolT}, for every finite set $F\subset G$, one has
\[
\max_{g,g'\in F}\psi(gh_\ell g')\to\infty
\]
as $\ell\to\infty$. As the set of vectors $\{\pi_n(g)\xi_n \colon g\in G, n\geq 1\}$ is total in $\Hh$, it suffices, by Proposition 2.3, to prove that, for all $g_1,\ldots,g_m\in G$, for every $n\geq 1$ and for every $\ep>0$ there exists $\ell\geq 1$ such that $|\langle\pi_n(g_jh_\ell g_k)\xi_n|\xi_n\rangle|<\ep$ for all $j,k$. But
\[
\langle\pi_n(g_jh_\ell g_k)\xi_n|\xi_n\rangle=\exp(-\psi(g_jh_\ell g_k)/n)\to 0
\]
as $\ell \to \infty$.
\end{proof}

In order to define $(S,\nu)$, we choose a countable orthonormal basis $\mathfrak{B}$ of $\Hh'$ which contains $\{\xi_n \colon n\geq 1\}$, we set
\[
(S,\nu):=\prod_{b\in \mathfrak B}\Big(\R,\frac{1}{\sqrt{2\pi}}\exp\Big(-\frac{x^2}{2}\Big)dx\Big)
\]
and we define the random variable $X_b:S\rightarrow\R$ by $X_b((\om_{b'})_{b'\in\mathfrak B})=\om_b$ for every $b$.
Then the map 
\[
\xi=\sum_{b\in\mathfrak B}\xi_b b\mapsto \sum_{b\in\mathfrak B}\xi_bX_b
\]
from $\Hh'$ to $L^2(S,\nu)$ extends to a surjective isometry $u:\Hh^\sigma\rightarrow L^2(S,\nu)$ which sends $\Hh^{\times 0}$ onto the space of constant functions on $S$ and such that
\[
u\Big(\sum_{s\in S_n}b_{s(1)}\otimes\ldots\otimes b_{s(n)}\Big)=n!X_{b_1}\ldots X_{b_n}
\]
for all $b_1,\ldots,b_n\in\mathfrak B$. Moreover, there is a $\nu$-preserving action of $G$ on $(S,\nu)$ such that $u^*\pi_S(g)u=\pi^\sigma(g)$ and $u^*\pi_{S,0}(g)u=\pi_0^\sigma(g)$ for all $g\in G$. In particular, Lemma 3.2 implies that the action of $H$ on $(S,\nu)$ is weakly mixing. Finally, the non-trivial asymptotically invariant sequence $(A_n)\subset \B_S$ of condition (b) above is obtained by setting
$A_n=\{\om\in S\colon X_{\xi_n}(\om)\geq 0\}$. See pages 23 and 24 of \cite{ccjjv} for further details.

\par\vspace{2mm}
Let us describe now our construction of the dynamical system $(\Om,\B,\mu,G)$ which is taken from \cite{DJZ}.

Let $(A_n)_{n\geq 1}\subset \B_S$ be the above non-trivial asymptotically invariant sequence. Then the following inequalities
\begin{align*}
|\nu(gA_n\cap A_n)-1/2|
&=
|\nu(gA_n\cap A_n)-\nu(A_n)|\\
&=
\int_S\chi_{A_n}(\chi_{gA_n}-\chi_{A_n})d\nu\\
&\leq
\int_S|\chi_{gA_n}-\chi_{A_n}|d\nu\\
&=
\nu(gA_n\bigtriangleup A_n)
\end{align*}
show that, for all positive integers $m$ and $k$, there exists an integer $n(k,m)$ such that 
\[
\sup_{g\in K_m}\left|\nu(gA_{n(k,m)} \cap A_{n(k,m)}) - \frac{1}{2}\right|\le \frac{1}{2}\left(1-e^{-\frac{1}{m2^{k}}}\right).
\]
Then set 
\[
B_m:=\prod_{k\geq 1}A_{n(k,m)}\subset X:=\prod_{k\geq 1}S
\]
for every $m>0$.

We equip the set $X$ with a $\sigma$-algebra $\Cc$ containing $(B_m)_{m\geq 1}$ and with a measure $\mu:\Cc\rightarrow [0,\infty]$ which have the following properties (see \cite{DJZ}, Propositions 2.8, 3.4 and 3.5):
\begin{enumerate}
\item [(i)] The $\sigma$-algebra $\Cc$ is generated by the collection (denoted by $\F_{c,0}$ in Definition 3.1 of \cite{DJZ}) of sets $C=\prod_n C_n$ such that $\prod_n 2\nu(C_n):=\lim_{N\to\infty}\prod_{n=1}^N 2\nu(C_n)$ exists in $[0,\infty)$ and such that $\prod_n 2\nu(C_n)=0$ or 
\begin{equation}
\lim_{N\to\infty}\prod_{n=N}^\infty 2\nu(gC_n\cap C_n)=1
\end{equation}
uniformly for $g\in K_1$.
\item [(ii)] The measure $\mu$ on $\Cc$ satisfies
\[
\mu(C)=\prod_n 2\nu(C_n)
\]
for every set $C=\prod_n C_n\in\F_{c,0}$.
\item [(iii)] The diagonal action of $G$ on $X$ is $\Cc$-measurable.
\item [(iv)] For every $A\in\Cc$ such that $\mu(A)<\infty$, one has
\[
\lim_{g\to e}\mu(gA\bigtriangleup A)=0.
\]
\end{enumerate}

\begin{remark}
\begin{enumerate}
\item [(1)] The measure $\mu$ on $X$ is not necessarily $\sigma$-finite, as is proved in Proposition 2.6 of \cite{DJZ}. 
Thus, equality (3.1) in (i) and property (iv) are continuity properties that are needed to restrict $\mu$ to a $G$-invariant subset $\Om\in\Cc$ on which is it $\sigma$-finite: see Section 3 of \cite{DJZ}.
\item [(2)] We observe that the sequence $(B_m)_{m\geq 1}$ is contained in $\Cc$: indeed, as 
\[
e^{-\frac{1}{m2^k}}\leq 2\nu(gA_{n(k,m)}\cap A_{n(k,m)})\leq 1
\]
for every $g\in K_m$ and for every $k\geq 1$, we get that
\begin{align*}
1
&\ge 
\lim_{N\rightarrow \infty } \prod_{k=N}^{\infty} 2\nu(gA_{n(k,m)} \cap A_{n(k,m)}) \\
&\ge 
\lim_{N\rightarrow \infty } \prod_{k=N}^{\infty} e^{-\frac{1}{m2^{k}}} \\
&= 
e^{-\frac{1}{m} \lim_{N\rightarrow \infty } \sum_{k=N}^{\infty} \frac{1}{2^{k}}}=1.
\end{align*}
In particular, it converges uniformly for $g\in K_1$.
\end{enumerate}
\end{remark}

The following proposition is inspired by Proposition 2.8 of \cite{DJZ}.

\begin{proposition}
The dynamical system $(X,\Cc,\mu,H)$ is weakly $C_0$, namely, for all $A,B\in \Cc$ such that $\mu(A),\mu(B)<\infty$ and for every $\ep>0$, there exists $h\in H$ such that $\mu(hA\cap B)<\ep$.
\end{proposition}
\begin{proof}
Assume first that $A,B\in\F_{c,0}$ have positive measures, and write $A=\prod_nA_n$ and $B=\prod_nB_n$, so that 
\[
\mu(A)=\prod_{n \ge 1} 2 \nu(A_n) \quad \text{and} \quad \mu(B)=\prod_{n \ge 1} 2 \nu(B_{n}).
\]
Let $\ep>0$ be fixed and take  $\ep'>0$ small enough in order that 
\[
\delta :=\frac{1}{2}+ \ep' +\frac{\ep'}{\frac{1}{2}- \ep'}<1.
\]
Since $0<\mu(A),\mu(B)<\infty$, there exists $N$ large enough such that 
\[
\frac{1}{2}-\ep'<\nu(A_{n}),\nu(B_{n})<\frac{1}{2}+\ep' \quad \forall n\ge N.
\]
Since $\delta < 1$, there exists $m$ large enough such that $\delta^{m+1}<\ep/\mu(A)$. The action of $H$ on $(S,\nu)$ being weakly mixing, by Lemma 2.5,
there exists $h\in H$ such that
\[
|\nu(hA_{n}\cap B_{n})-\nu(A_{n})\nu(B_{n})|\le \ep'
\]
for all $n\in\{N,\ldots,N+m\}$.
Then we have
\begin{align*}
\mu(hA\cap B)
&\leq 
\prod_{n=1}^{N-1} 2\nu(A_{n})\cdot\prod_{n=N}^{N+m}2(\nu(A_{n})\nu(B_{n})+\ep')\cdot\prod_{n\ge N+m+1}2\nu(A_{n})\\
&= 
\mu(A) \cdot \prod_{n=N}^{N+m} \frac{2\nu(A_{n})\nu(B_{n})+2\ep'}{2\nu(A_{n})}\\
&=
\mu(A)\prod_{n=N}^{N+m} \left(\nu(B_{n})+\frac{\ep'}{\nu(A_{n})}\right)\\
&< 
\mu(A)\prod_{n=N}^{N+m} \left(\frac{1}{2}+ \ep' +\frac{\ep'}{\frac{1}{2}- \ep'}\right)\\
&=
\mu(A)\delta^{m+1}<\ep.
\end{align*}
Hence the claim holds for all sets $A,B\in\F_{c,0}$.

The same claim holds for $A,B$ which belong to the semiring $\F_c$ generated by $\F_{c,0}$ because, by Proposition 3.4 of \cite{DJZ}, for all $A,B\in\F_c$, there exist $C,D\in\F_{c,0}$ such that $A\subset C$ and $B\subset D$.

Moreover, it also holds for all elements of the ring $\Rr(\F_c)$ generated by $\F_c$. 

Finally, if $A,B\in\Cc$ are such that $0<\mu(A),\mu(B)<\infty$, if $\ep>0$ is given, there exist two sequences $(C_k)_{k\geq 1},(D_\ell)_{\ell\ge 1}\subset \Rr(\F_c)$ such that 
\[
A\subset \bigcup_{k\geq 1}C_k\quad\text{and}\quad B\subset\bigcup_{\ell\ge 1}D_\ell
\]
and
\[
\mu(A)\leq \sum_k\mu(C_k)<\mu(A)+\ep\quad\text{and}\quad \mu(B)\le\sum_\ell \mu(D_\ell)<\mu(B)+\ep.
\]
Choose first $N$ large enough so that $\sum_{\ell>N}\mu(D_\ell)<\ep/3$. Then, as
\[
gA\cap B\subset \left(\bigcup_{\ell=1}^N gA\cap D_\ell\right) \cup \left(\bigcup_{\ell>N}D_\ell\right),
\]
we get
\[
\mu(gA\cap B)\le \sum_{\ell=1}^N\mu(gA\cap D_\ell)+\ep/3
\]
for every $g\in G$.

Choose next $M$ large enough so that $\sum_{k>M}\mu(C_k)<\ep/3N$. Then, as
\[
gA\cap D_\ell\subset \left(\bigcup_{k=1}^M gC_k\cap D_\ell\right)\cup\left(\bigcup_{k>M}gC_k\right)
\]
for every $1\leq\ell\leq N$, and since $\mu$ is $G$-invariant, we get
\[
\mu(gA\cap B)\leq \sum_{\ell=1}^N\sum_{k=1}^M\mu(gC_k\cap D_\ell)+2\ep/3
\]
for every $g\in G$. By the previous part of the proof applied to $C:=\bigcup_{k=1}^M C_k$ and $D:=\bigcup_{\ell=1}^N D_\ell$, there exists $h\in H$ such that 
\[
\mu(hA\cap B)<\ep.
\]
\end{proof}

The proof of Theorem \ref{mainthm} will be complete if we prove that there is a $G$-invariant set $\Om\in\Cc$ on which $\mu$ is $\sigma$-finite.

As in Proposition 3.8 of \cite{DJZ}, we fix a countable dense subset $D\ni e$ of $G$, we set $Y=\bigcup_{m\geq 1}B_m$ and
\[
\Om=\bigcup_{h\in D}hY
\]
on which $\mu$ is obviously $\sigma$-finite.

The continuity condition (iv) implies that $\Om$ is $G$-invariant (see the proof of Proposition 3.8 of \cite{DJZ}),
thus the proof of Theorem \ref{mainthm} is complete.

\par\vspace{3mm}

\bibliographystyle{plain}
\bibliography{refFaible}

\end{document}